\documentclass{elsarticle}
\usepackage{graphicx}
\usepackage{hyperref}
\usepackage{amsmath, amssymb, amsfonts, amsthm, enumerate}
\usepackage{wrapfig}
\usepackage{a4wide}

\date{}

\newtheorem{theorem}{\bf Theorem}
\newtheorem{claim}[theorem]{\bf Claim}
\newtheorem{lemma}[theorem]{\bf Lemma}
\newtheorem{conjecture}[theorem]{\bf Conjecture}

\newtheorem{observation}[theorem]{\bf Observation}
\newtheorem{definition}[theorem]{\bf Definition}

\begin{document}

\begin{frontmatter}
\journal{a journal}
\title{Linearly many rainbow trees in properly edge-coloured complete graphs}
\author{Alexey Pokrovskiy\fnref{label5}
}
\ead{dr.alexey.pokrovskiy@gmail.com}
\address{Department of Mathematics, ETH, 8092 Zurich, Switzerland.}
\fntext[label5]{Research supported in part by SNSF grant 200021-175573.}

\author{
Benny Sudakov\fnref{label5}
}
\ead{ benjamin.sudakov@math.ethz.ch}

\address{Department of Mathematics, ETH, 8092 Zurich, Switzerland.}

\begin{abstract}
A subgraph of an edge-coloured complete graph is called rainbow if all its edges have different colours. 
The study of rainbow decompositions has a long history, going back to the work of Euler on Latin squares.
In this paper we discuss three problems about decomposing complete graphs into rainbow trees: 
the Brualdi-Hollingsworth Conjecture, Constantine's Conjecture, and the Kaneko-Kano-Suzuki Conjecture.
We show that in every proper edge-colouring of $K_n$ there are $10^{-6}n$ 
edge-disjoint spanning isomorphic rainbow trees.  This simultaneously improves the best known bounds on all these conjectures. Using our method we also show that 
every properly $(n-1)$-edge-coloured $K_n$  has $n/9-6$ edge-disjoint rainbow trees, giving further improvement on the Brualdi-Hollingsworth Conjecture.
\end{abstract}

\begin{keyword}
Rainbow trees, proper edge-colourings, graph decompositions.
\end{keyword}
\end{frontmatter}

\section{Introduction}
In this paper we consider the following question: Can the edges of every properly edge-coloured complete graph be decomposed into edge-disjoint rainbow spanning trees. Here a properly edge-coloured complete graph $K_n$ means an assignment of colours to the edges of $K_n$ so that no two edges at a vertex receive the same colour. A rainbow spanning tree in $K_n$ is a tree containing every vertes of $K_n$, all of whose edges have different colours.

The study of rainbow decompositions dates back to the 18th century when Euler studied the question ``for which $n$ does there exist a properly $n$-edge-coloured $K_{n,n}$ which can be decomposed into $n$ edge-disjoint rainbow perfect matchings \footnote{Euler studied the values of $n$ for which a pair of $n\times n$ orthogonal Latin squares exists. Using a standard argument, it is easy to show that $n\times n$ orthogonal Latin squares are equivalent objects to rainbow perfect matching decompositions of $K_{n,n}$.}.''
Euler constructed such proper $n$-edge-colourings of $K_{n,n}$ whenever $n\not\equiv 2 \pmod{4}$, and conjectured that these are the only values of $n$ for which they can exist. The $n=6$ case of this conjecture is Euler's famous ``36 officers problem'', which was eventually proved by Tarry in $1901$. For larger $n$, Euler's Conjecture was disproved in 1959 by  Parker, Bose, and Shrikhande. 
Together these results give a complete description of the values of $n$ for which there exists a properly $n$-edge-coloured $K_{n,n}$ which can be decomposed into $n$ edge-disjoint rainbow perfect matchings.

Decompositions of properly $(2n-1)$-edge-coloured $K_{2n}$ into edge-disjoint rainbow perfect matchings have also been studied. They were introduced by Room in 1955 \footnote{Room actually introduced objects which are now called ``Room squares''. It is easy to show that Room squares are equivalent objects to decompositions of $(2n-1)$-edge-coloured $K_{2n}$ into edge-disjoint rainbow perfect matchings.}, who raised the question of which $n$ they exist for. Wallis showed that such decompositions of $K_{2n}$ exist if, and only if, $n\neq 2$ or $4$.
Rainbow perfect matching decompositions of both $K_{n,n}$ and $K_{2n}$ have  found applications in scheduling tournaments and constructing experimental designs (see eg \cite{Designs_Survey}). 

Euler and Room wanted to determine the values of $n$ for which there exist  colourings of $K_{n,n}$ or $K_{n}$ with rainbow matching decompositions. However given an \emph{arbitrary} proper edge-colouring of $K_{n,n}$ or $K_{n}$ it is not the case that it must have a decomposition into rainbow perfect matchings. A natural way of getting around this is to consider decompositions into rainbow graphs other than perfect matchings. In the past decompositions into rainbow subgraphs such as cycles and triangle factors have been considered~\cite{Colbourn_Mendelsohn}.

An additional reason to study rainbow subgraphs arises in Ramsey theory, more precisely in the
canonical version of Ramsey's theorem, proved by Erd\H{o}s and Rado \cite{ER-canonical} in 1950 . Here
the goal is to show that edge-colourings of $K_n$, in which each colour
appears only few times contain rainbow copies of certain graphs (see,
e.g., introduction of \cite{Sudakov_Volec}, for more details). 

In this paper we consider decompositions into rainbow trees. In contrast to the perfect matching case, it is believed that every properly edge coloured $K_n$ can be decomposed into edge-disjoint rainbow trees. 
 This was conjectured by three different sets of authors.
\begin{conjecture}[Brualdi and Hollingsworth, \cite{Brualdi_Hollingsworth}]\label{Conjecture_Brualdi_Hollongsworth}
Every properly $(2n-1)$-edge-coloured $K_{2n}$ can be decomposed into edge-disjoint rainbow spanning trees.
\end{conjecture}
\begin{conjecture}[Kaneko, Kano, and Suzuki, \cite{Kaneko_Kano_Suzuki}]\label{Conjecture_Kaneko_Kano_Suzuki}
Every properly edge-coloured $K_{n}$ contains $\lfloor n/2\rfloor$ edge-disjoint isomorphic rainbow spanning trees.
\end{conjecture}
\begin{conjecture}[Constantine, \cite{Constantine}]\label{Conjecture_Constantine}
Every properly $(2n-1)$-edge-coloured $K_{2n}$ can be decomposed into edge-disjoint isomorphic rainbow spanning trees.
\end{conjecture}
There are many partial results on the above conjectures. It is easy to see that every properly coloured $K_n$ contains a single rainbow tree---specifically the star at any vertex will always be rainbow. Strengthening this, various authors have shown that more disjoint trees exist under assumptions of Conjectures~\ref{Conjecture_Brualdi_Hollongsworth}--\ref{Conjecture_Constantine}.

Brualdi and Hollingsworth~\cite{Brualdi_Hollingsworth} showed that every properly $(2n-1)$-coloured $K_{2n}$ has $2$ edge-disjoint rainbow spanning trees. 
Krussel, Marshall, and Verrall~\cite{Krussel_Marshall_Verrall} showed that there are $3$ rainbow spanning trees under the same assumption.
Kaneko, Kano, and Suzuki~\cite{Kaneko_Kano_Suzuki} showed that $3$ edge-disjoint rainbow spanning trees exist in any proper colouring of $K_n$ (with any number of colours).
Akbari and Alipour~\cite{Akbari_Alipour} showed that $2$ edge-disjoint rainbow spanning trees exist in any colouring of $K_n$ with at most $n/2$ edges of each colour.
Carraher, Hartke, and Horn~\cite{Carraher_Hartke_Horn} showed that under the same assumption, $\lfloor n/1000\log n\rfloor$ edge-disjoint rainbow spanning trees exist. In particular this implies that every properly coloured $K_n$ has this many edge-disjoint spanning rainbow trees.
Horn~\cite{Horn} showed that there is an $\epsilon>0$ such that every  $(2n-1)$-coloured $K_{2n}$ has $\epsilon n$ edge-disjoint rainbow spanning trees. 
Subsequently, Fu, Lo, Perry, and Rodger~\cite{Fu_Lo_Perry_Rodger} showed that every $(2n-1)$-coloured $K_{2n}$ has $\lfloor\sqrt{6n+9}/3\rfloor$ edge-disjoint rainbow spanning trees. 
For Conjecture~\ref{Conjecture_Constantine}, Fu and Lo~\cite{Fu_Lo} showed that  every $(2n-1)$-coloured $K_{2n}$ has $3$ isomorphic edge-disjoint spannind trees. In addition to these results, there has been a fair ammount of work showing that edge-coloured complete graphs with certain specific colourings can be decomposed into rainbow spanning trees (see eg \cite{Akbari_Alipour_Fu_Lo, Constantine}). 

Here is a summary of the the best known results for these problems for large $n$. Horn proved for the Brualdi-Hollingsworth Conjecture that $\epsilon n$ edge-disjoint rainbow spanning trees exist.
For the Kaneko-Kano-Suzuki Conjecture,  Carraher, Hartke, and Horn proved that $\lfloor n/1000\log n\rfloor$ edge-disjoint rainbow spanning trees exist.
For Constantine's Conjecture, Fu and Lo proved that $3$ edge-disjoint rainbow spanning trees exist.

Here we substantially improve the best known bounds for all three conjectures. Define a $t$-spider to be a tree obtained from a star by subdividing $t$ of its edges once. We prove the following.
\begin{theorem}\label{TheoremProperColouring}
Every properly edge-coloured $K_n$  contains $10^{-6} n$ edge-disjoint rainbow spanning $t$-spiders for any $0.0007n\leq t\leq 0.2n$.
\end{theorem}

Beyond improving the bounds on Conjectures~\ref{Conjecture_Brualdi_Hollongsworth}--\ref{Conjecture_Constantine}, Theorem~\ref{TheoremProperColouring} is qualitatively stronger than all of them. Firstly, the isomorphism class of the spanning trees in Theorem~\ref{TheoremProperColouring} is independent of the colouring on $K_n$ (whereas Constantine's Conjecture allows for such a dependency). Additionally Theorem~\ref{TheoremProperColouring} produces isomorphic spanning trees under a weaker assumption than Constantine's Conjecture (namely we do not specify that $K_n$ is $(n-1)$-coloured). 

The method we use to prove Theorem~\ref{TheoremProperColouring} is quite flexible. For any one of the three conjectures, it is easy to modify our method to give a further improvement on the $10^{-6} n$ bound from our theorem. 
In order to illustrate this, we will show that in the case of the Brualdi-Hollingsworth Conjecture one can cover over 20\% of the edges by spanning rainbow trees. 
\begin{theorem}\label{Theorem1Factorization}
Every properly $(n-1)$-edge-coloured  $K_n$ has $n/9-6$ edge-disjoint spanning rainbow trees. 
\end{theorem}

\subsection*{Notation}
Throughout the paper all colourings of graphs will be edge-colourings. 
For an edge $e$, we use $c(e)$ to denote the colour of $e$.
For a colour $c$ and a graph $G$, we will use ``$c\in G$'' to mean that $G$ has a colour $c$ edge.

For a graph $G$ and a set of vertices $U$ we use $G\setminus U$ to denote the induced subgraph of $G$ on $V(G)\setminus U$.
For a graph $G$ and a set of edges $E$ we use $G\setminus E$ to denote the subgraph  of $G$ formed by deleting the edges in $E$. Thus for a subgraph $H$ of $G$, ``$G\setminus V(H)$'' and ``$G\setminus E(H)$'' denote the subgraphs of $G$ formed by deleting the vertices and edges of $H$ respecively.

\begin{definition}
A graph $S$ is a $t$-spider if $V(S)=\{r, j_1, \dots, j_t, x_1, \dots, x_t, y_1, \dots, y_{|S|-2t-1}\}$ with $E(S)=\{rj_1, \dots, rj_t\}\cup\{ry_1, \dots, ry_{|S|-2t-1}\}\cup\{j_1x_1, \dots, j_tx_t\}$.
\end{definition}
The vertex $r$ is called the \emph{root} of the spider. 
The vertices $y_1, \dots, y_{|S|-2t-1}$ are called \emph{ordinary leaves}.
We will use ``$D$ is a $(\leq t)$-spider'' to mean that ``$D$ is a $s$-spider for some $s\leq t$.'' We will often use the following two simple observations  to build spiders.

\begin{observation}\label{ObservationStarPlusMatching}
Let $S$ be a star rooted at $r$ and $M$ be a matching with $|e\cap S|=1$ and $r\not\in e$ for all the edges $e\in M$. Then $M\cup S$ is an $|M|$-spider.
\end{observation}

\begin{observation}\label{ObservationTwoDoubleStars}
Let $D_1$ be a $d_1$-spider rooted at $r$, and  $D_2$  a $d_2$-spider rooted at $r$ with $V(D_1)\cap V(D_2)=\{r\}$. Then $D_1\cup D_2$ is a $(d_1+d_2)$-spider.
\end{observation}

\section{Proof sketch}\label{SectionProofSketch}
In this self-contained section we give a sketch of the proof of Theorem~\ref{TheoremProperColouring}. Throughout the section, we fix a properly coloured complete graph  $K_n$ and let $m=10^{-6} n$ be the number of edge-disjoint spiders we are trying to find. 

For the purposes of this proof sketch, it is convenient to introduce some notation.
We say that a family of spiders $\mathcal D= \{D_1, \dots, D_m\}$ is \emph{root-covering} if the root of $D_i$ is in $V(D_j)$ for any $i,j\in \{1, \dots, m\}$. 
The basic idea of the proof of Theorem~\ref{TheoremProperColouring} is to first find a root-covering family of non-spanning, non-isomorphic, spiders $\mathcal D= \{D_1, \dots, D_m\}$. Then, for each $i$, the spider $D_i$ is modified into a spanning, isomorphic rainbow spider. 
The reason for considering root-covering families is that the roots are the highest degree vertices in spiders. Because of this, they are intuitively the most difficult vertices to cover in the spiders we are looking for. Thus in the proof we first find a family of spiders which is root-covering, and then worry about making them spanning and isomorphic. 

The proof of Theorem~\ref{TheoremProperColouring} naturally splits into three steps:
\begin{enumerate}[(1)]
\item Find a {root-covering} family of large edge-disjoint rainbow spiders $D_1, \dots, D_{m}$ in $K_{n}$.
\item Modify the spiders from (1) into a root-covering family of \emph{spanning}, edge-disjoint, rainbow spiders $D'_1, \dots, D'_{m}$.
\item Modify the spiders from (2) into a root-covering family of spanning, edge-disjoint, rainbow, \emph{isomorphic} spiders $D''_1, \dots, D''_{m}$.
\end{enumerate}

Step (1) is the easiest part of the proof. To prove it, we first find a family of disjoint rainbow \emph{stars} $S_1, \dots, S_{m}$ rooted at $r_1, \dots, r_{m}$ in $K_n$. 
Then by exchanging some edges between these stars, we obtain spiders $D_1, \dots, D_{m}$ rooted at $r_1, \dots, r_{m}$ which is root-covering. See Lemma~\ref{LemmaCoverBipartiteGraph}.

Step (2) is the hardest part of the proof. It involves going through the spiders  $D_1, \dots, D_{m}$ from part (1) one by one and modifying them. For each $i$, we modify $D_i$ into a spanning spider $D_i'$ with $D_i'$  edge-disjoint from the spiders $D'_1, \dots, D'_{i-1},$ $D_{i+1}, \dots, D_{m}$ and $D_i'$ having the same root as $D_i$. In order to describe which edges we can use in $D_i'$, we make the following definition.
\begin{definition}
Let $\mathcal D= \{D_1, \dots, D_m\}$ be a family of edge-disjoint spiders in a coloured $K_n$.
Let $D_i=S_i\cup \hat D_i$ where $S_i$ is the star consisting of the ordinary leaves of $D_i$.
We let $G(D_i, \mathcal D)$ denote the subgraph of $K_n$ formed by deleting the following:
\begin{itemize}
\item All the roots  of the spiders $D_1, \dots, D_{i-1}, D_{i+1}, \dots, D_m$.
\item All the edges  of the spiders $D_1, \dots, D_{i-1}, D_{i+1}, \dots, D_m$.
\item All  edges  sharing a colour with $\hat D_i$.
\item All vertices of $\hat D_i$ except the root.
\end{itemize}
\end{definition}

 The intuition behind this definition is that we can freely modify $D_i$ using edges from $G(D_i, \mathcal D)$ without affecting the other spiders $D_1, \dots, D_{i-1}, D_{i+1}, \dots, D_m$. The following observation makes this precise.
\begin{observation}\label{ObservationExtendStar}
Let $\mathcal D= \{D_1, \dots, D_m\}$ be a family of rainbow  spiders in a coloured $K_n$. 
Let $D_i=S_i\cup \hat D_i$ where $S_i$ is the star consisting of the ordinary leaves of $D_i$. Then for any rainbow spider $\hat S_i$ in   $G(D_i, \mathcal D)$ with $S_i$ and  $\hat S_i$ having the same root, we have that $\hat S_i\cup \hat D_i$ is a rainbow spider in $K_n$. 

In addition  if $\mathcal D$ was edge-disjoint and root-covering, then $\mathcal D\setminus \{D_i\}\cup\{\hat S_i\cup \hat D_i\}$ is edge-disjoint and root-covering.
\end{observation}
 
 A crucial feature of $G(D_i, \mathcal D)$ is that it has high minimum degree. 
\begin{observation}\label{ObservationHighMinimumDegree}
For a family of spiders $\mathcal D= \{D_1, \dots, D_m\}$ in a properly coloured $K_n$ with $D_i$ a $t$-spider we have $\delta(G(D_i, \mathcal D))\geq n-3m-4t-1$.
\end{observation}

To solve step (2) we consider the graph $G(D_i, \mathcal D)$ for $\mathcal D=\{D'_1, \dots, D'_{i-1}, D_i, \dots, D_{m}\}$. Using Observation~\ref{ObservationExtendStar} to solve (2) it is enough to find a spanning rainbow spider $D'_i$ in $G(D_i, \mathcal D)$ having the same root as $D_i$.
From Observation~\ref{ObservationHighMinimumDegree} we know that $G(D_i, \mathcal D)$ has high minimum degree.
Thus, to solve (2) it would be sufficient to show that ``every properly coloured graph with high minimum degree and a vertex $r$  has a spanning rainbow spider rooted at $r$.'' Unfortunately this isn't true since it is possible to have have a properly coloured graph $G$ with high minimum degree which has less than $|G|-1$ colours (and hence has no spanning rainbow tree).  

However, in a sense, ``having too few colours'' is the only barrier to finding a spanning rainbow spider in a high minimum degree graph. Lemmas~\ref{LemmaAddOneVertexToStar} and~\ref{LemmaAddManyVerticesToStar} will show that as long as there are enough edges of colours not touching $r$, then it is possible to find a spanning rainbow spider rooted at $r$ in a high minimum degree graph. This turns out to be sufficient to complete the proof of step (2) since it is possible to ensure that the graphs $G(D_i, \mathcal D)$ have a lot of edges of colours outside $D_i$. The details of this are somewhat complicated and explained in Section~\ref{SectionIsomorphicTrees}.

Step (3) is similar in spirit to step (2). It consists of going through the spiders  $D'_1, \dots, D'_{m}$ one by one, and modifying $D'_i$ into a spanning spider $D_i''$ with $D_i''$  edge-disjoint from the spiders $D''_1, \dots, D''_{i-1},$ $D'_{i+1}, \dots, D'_{m}$ and $D''_i$ having the same root as $D'_i$. We once again consider the graph  $G(D'_i, \mathcal D)$ for $\mathcal D=\{D''_1, \dots, D''_{i-1}, D'_{i}, \dots, D'_m\}$ and notice that it has high degree. Because of this, to prove  step (3) it is sufficient to show that ``in every properly coloured graph $G$ with high minimum degree and a spanning rainbow star $S$, there is a spanning rainbow $t$-spider for suitable $t$.'' This turns out to be true for $t\geq 3$, and is proved by replacing edges of $D'_i$ for suitable edges outside $D'_i$ (see Lemma~\ref{LemmaChangeDoubleStarParameter}).

\section{Many rainbow trees in 1-factorizations}

The proof of Theorem~\ref{Theorem1Factorization} naturally splits into two parts. In the first part we show that one can find large edge-disjoint rainbow trees $T_1, \dots, T_n$ with the property that any vertex in $V(T_i)\setminus V(T_j)$ has small degree in $T_i$. In the second part we modify the trees from the first part one by one into spanning trees.  The first part is summarized in the following lemma.
\begin{lemma}\label{Lemma1FactorizationCoverSet}
Let $m\equiv 1$ or $3 \pmod{6}$ and $n> 9m$. 
Let $G=K_{n}\setminus E(K_{n-m})$ be properly coloured with $n-1$ colours with $V(G)=A\dot\cup B$ where $B$ is the copy of $K_{n-m}$ and $|A|=m$.
Then $G$ has edge-disjoint rainbow $\left(\frac{m-1}{2}\right)$-spiders $D_1, \dots, D_{m}$  of order $n-(m-1)/2$ with each $D_i$ rooted in $A$ and covering all the vertices in $A$.
\end{lemma}
\begin{proof}
Recall that a Steiner triple system is a $3$-uniform hypergraph $\mathcal S$ with the property that for any pair of vertices $x,y\in V(\mathcal S)$, there is precisely one edge in $\mathcal S$  containing both $x$ and $y$. It is well known that a Steiner triple system with $m$ vertices if, and only if, $m \equiv 1$ or $3 \pmod{6}$.
Therefore, we can choose a  Steiner triple system $\mathcal S$ with vertex set $A$ (which exists since  $|A|=m\equiv 1$ or $3 \pmod{6}$). For a vertex $x\in A$ and a colour $c$,  let $v(x,c)$ be the unique vertex $v$ with $c(xv)=c$.

Choose a cyclic orientation $(x,y,z)$ for each $\{x,y,z\}\in \mathcal{S}$. Formally, this a family of ordered triples $\vec{\mathcal S}\subseteq A\times A\times A$ where for every $\{x,y,z\}\in \mathcal{S}$ we either have $(x,y,z), (y,z,x), (z,x,y)\in \vec{\mathcal S}$ or $(z,y,x), (x,z,y), (y,x,z)\in \vec{\mathcal S}$ (but not both).

\begin{claim}\label{ClaimAssignVerticesToTriples}
To every triple $(x,y,z)\in \vec{\mathcal S}$, we can assign a vertex $b(x,y,z)$ with the following properties.
\begin{enumerate}[(i)]
\item $c(xb(x,y,z))=c(yb(y,z,x))=c(zb(z,x,y))$.
\item $b(x,y,z)\in B$.
\item $b(x,y,z)\neq b(x, u, v), b(v,x,u), b(y,u,v), b(v,y,u)$ for any $u$ and $v$.
\end{enumerate}
\end{claim}
\begin{proof}
To produce such an assignment, we go through every triple $(x,y,z)\in \vec {\mathcal{S}}$ and choose vertices $b(x,y,z)$, $b(y,z,x)$, $b(z,x,y)$ satisfying (i) -- (iii) with respect to the previously chosen vertices. 

Since $b(x,y,z)$, $b(y,z,x)$, $b(z,x,y)$ need to satisfy (i), notice that we must have $b(x,y,z)=v(x,c)$, $b(y,z,x)=v(y,c)$, $b(z,x,y)=v(z,c)$ for some colour $c$. 
Therefore, we just need to choose some colour $c$ for which (ii) -- (iii) hold with the choice $b(x,y,z)=v(x,c)$, $b(y,z,x)=v(y,c)$, $b(z,x,y)=v(z,c)$.

We claim that there are at most $3(|A|-1)$ colours $c$ for which $b(x,y,z)=v(x,c)$ wouldn't satisfy (ii) and (iii) with respect to the previously chosen vertices.
There are  $|A|-1$ colours  for which $v(x,c)\in A$, and hence $|A|-1$ colours for which (ii) doesn't hold for $b(x,y,z)=v(x,c)$.
There are $(|A|-1)/2$ triples $\{x,u,v\}\in \mathcal S$  containing $x$, and hence $(|A|-1)/2$ ordered triples of the form $(x, u, v)\in \vec{\mathcal S}$ for $u,v\in A$. This shows that there are at most $(|A|-1)/2$ colours for which $v(x,c)$ could equal  $b(x, u, v)$ for a previously chosen vertex. Similarly, there are at most $(|A|-1)/2$  colours for which $v(x,c)$ could equal each of  $b(v,x,u), b(y,u,v), b(v,y,u)$ for a previously chosen vertex. In total this gives at most $4\cdot (|A|-1)/2$ colours for which  (iii) might not hold for $b(x,y,z)=v(x,c)$ with respect to the previously chosen vertices.

By symmetry, we have that there are at most $3(|A|-1)$ colours $c$ for which $b(y,z,x)=v(y,c)$ wouldn't satisfy (ii) and (iii), and   at most $3(|A|-1)$ colours $c$ for which $b(z,x,y)=v(z,c)$ wouldn't satisfy (ii) and (iii). In total this shows that there are at most $9(|A|-1)$ colours for which any of  $b(x,y,z)=v(x,c)$, $b(y,z,x)=v(y,c)$, $b(z,x,y)=v(z,c)$ might not satisfy (ii) and (iii) with respect to the previously chosen vertices.  Since the number of colours is $n-1\geq 9m-1> 9(|A|-1)$, there is some colour $c$ which we can choose  so that $b(x,y,z)=v(x,c)$, $b(y,z,x)=v(y,c)$, $b(z,x,y)=v(z,c)$ satisfy (i), (ii), and (iii).
\end{proof}

Let $A=\{1, \dots, {m}\}$. For $x=1, \dots, m$, define 
\begin{align*}
D_x^1&= \{zb(z,x,y): (z,x,y)\in \vec{\mathcal S}\}\\
D_x^2&= \{xy: (x,y,z)\in \vec{\mathcal S}\}\\
D_x^3&= \{xb: \text{$b\in B$ and $b\neq b(x,y,z)$ for $(x,y,z)\in \vec{\mathcal S}$}\}\\
D_x&=D_x^1\cup D_x^2\cup D_x^3.
\end{align*}

\begin{figure}
\centering
    \includegraphics[width=0.6\textwidth]{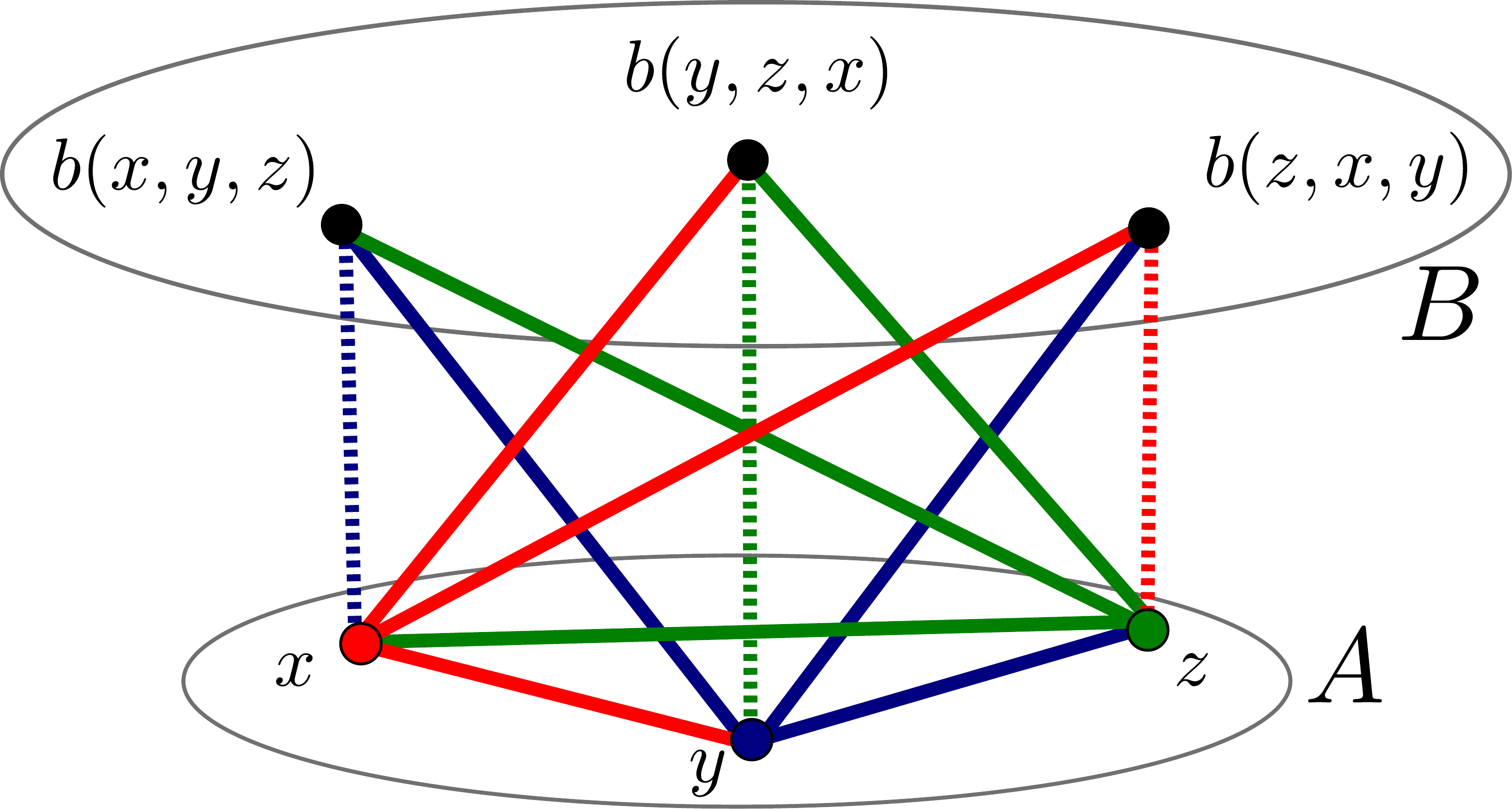}
  \caption{\small How edges involving $x, y, z$ are distributed between the spiders $D_x, D_y$, and $D_z$ for a triple $(x,y,z)\in \vec{\mathcal S}$. Here the colours do not correspond to the colours of edges in $G$, but rather to the three spiders $D_x, D_y$, and $D_z$: red edges are in $D_x$, blue edges are in $D_y$, and green edges are in $D_z$. The three dashed edges all have the same colour in $G$ as a consequence of Claim~\ref{ClaimAssignVerticesToTriples} (i). \label{Figure1FactorizationPartition}} 
\end{figure}

See Figure~\ref{Figure1FactorizationPartition} to see how the spiders $D_x, D_y$, and $D_z$  look for a triple  $(x,y,z)\in \vec{\mathcal S}$. 
We claim that $D_1, \dots, D_{m}$ satisfy the conditions of the lemma. 
To see that $D_x$ is rainbow, notice that using (i), the colours in $D_x$ are exactly the colours in the star in $G$ containing $x$ (which are all different since $G$ is properly coloured).
We have $e(D_x^1)=e(D_x^2)=(|A|-1)/2$ and $e(D_x^3)=|B|-(|A|-1)/2$ which implies that $e(D_x)=|B|+(|A|-1)/2=n-(m+1)/2$ as required.

To see that $D_x$ is a $(\frac{m-1}2)$-spider, first notice that  $D_x^2\cup D_x^3$ is a star.
Next notice that $D_x^1$ is a matching since by (iii) we have $b(z,x,y)\neq b(z',x,y')$ for any distinct ordered triples $(z,x,y),(z',x,y')\in \vec{\mathcal S}$. Notice that for every edge $zb(z,x,y)\in D_x^1$ we have $|\{z,b(z,x,y)\}\cap D_x^2\cup D_x^3|=1$ 
(This is true because $z\not\in D_x^2\cup D_x^3$ and $b(z,x,y)\in D_x^3$. 
To see that $z\not\in D_x^3$, notice that $D_x^3\subseteq B\cup \{x\}$, $z\in A$, and $z\neq x$.
To see that $z\not\in  D_x^2$, notice that since $\vec{\mathcal S}$ is an oriented Steiner triple system containing $(z,x,y)$, we do not  have  $(x,z,y')\in \vec{\mathcal S}$ for any $y'\in A$.
To see that $b(z,x,y)\in D_x^3$ notice that $xb(z,x,y)\in D_x^3$ which holds since by (iii) we have $b(z,x,y)\neq  b(x,y',z')$ for any $y', z'$).
We have  that  $D_x^2\cup D_x^3$ is a star and $D_x^1$ is a matching with $|e\cap (D_x^2\cup D_x^3)|=1$ for $e\in D_x^1$.
Since  $e(D_x^1)=\frac{m-1}2$ and $x\not\in V(D_x^1)$, Observation~\ref{ObservationStarPlusMatching} implies that $D_x= D_x^1 \cup (D_x^2\cup D_x^3)$ is a $(\frac{m-1}2)$-spider as required.

To see that $D_x$ covers $A$, notice that since $\vec{\mathcal S}$ is an orientated Steiner triple system, for any $y\in A$ either $(y, x, z)\in \vec{\mathcal S}$ or $(x, y, z)\in \vec{\mathcal S}$ holds for some $z$. In the first case $yb(y, x, z)\in D^1_x$ and in the second case $xy\in D^2_x$.

It remains to show that $D_x$ and $D_y$ are edge-disjoint for $x\neq y$. 
We have that $D_x^2$ is edge-disjoint from $D_y^1\cup D_y^3$ since the edges in $D_x^2$  go from $A$ to $A$, while the edges in $D_y^1\cup D_y^3$  go from $A$ to $B$. 
Similarly we have that $D_x^1\cup D_x^3$ is edge-disjoint from $D_y^2$.
We have that $D_x^2$ is edge-disjoint from $D_y^2$ since we do not have $(x,y,z), (y,x,z') \in \vec{\mathcal S}$ for any $z, z'$ (since $\vec{\mathcal S}$ is an oriented Steiner triple system).
We have that $D_x^3$  is edge-disjoint from $D_y^3$ since edges in $D_x^3$  go from $x$ to $B$, whereas edges in $D_y^3$  go from $y$ to $B$.
We have that $D_x^1$ is edge-disjoint from $D_y^1$ since $b(z,x,u)\neq b(z,y,w)$ from (iii).
To see  that $D_x^1$ is edge-disjoint from $D_y^3$ notice that the only edge in $D_x^1$ passing through $y$ is $yb(y,x,w)$ for some $w$. However $yb(y,x,w)\not\in  D_y^3$ by definition of $D_y^3$.
By the same argument, we have that $D_y^1$ is edge-disjoint from $D_x^3$, completing our proof of  $D_x$ and $D_y$ being edge-disjoint.
\end{proof}
We remark that the above lemma actually gives a decomposition of all the edges of $G$ into  disjoint spiders.
Lemma~\ref{Lemma1FactorizationCoverSet} is combined with the following lemma which allows us to modify  a large rainbow spider into a spanning rainbow tree. 
\begin{lemma}\label{Lemma1FactorizationMakeSpanningTree}
Suppose that $\delta +5.5\alpha< 1$.
Let $G$ be a sufficiently large properly coloured graph on $n$ vertices with $n-1$ colours each having  at least $ (1-2\alpha)n/2$ edges. Let $D$ be a rainbow $(\leq \alpha n/2)$-spider in $G$  rooted at $r$ of order  at least $\left(1-\frac{\alpha}{2}\right)n$ such that every $v\not\in V(D)$ has $d(v)\geq (1-\delta)n$.
Then $G$ has a spanning rainbow tree $T$ with  $d_{T_i}(u)\leq 3$ for $u\neq r$.
\end{lemma}
\begin{proof}
Without loss of generality, we may suppose that $D$ has order exactly $\left(1-\frac{\alpha}{2}\right)n$.
Let the vertices of $G\setminus V(D)$ be labeled $1, \dots, \alpha n/2$. Since there are exactly $\alpha n/2$ colours outside $D$, we can associate a distinct colour $c_v\not\in D$ to every vertex $v\not\in V(D)$.

We define  trees $T_0, T_1, \dots, T_{\alpha n/2}$ with $V(T_i)=V(D)\cup\{1, \dots, i\}$.  They will have the following properties.
\begin{enumerate}[(i)]
\item $T_i$ is a rainbow tree with $V(T_i)=V(D)\cup\{1, \dots, i\}$ using colours in $D$, and $c_1, \dots, c_i$.
\item For $u\neq r$ we have $d_{T_i}(u)\leq 3.$
\item $T_i$ has at least $(1-\alpha )n-i$ leaves.
\item $T_i$ has at most $i$ vertices $w$ with $d_{T_i}(w)=3$. 
\end{enumerate}
Notice that if we can construct such a sequence then the tree $T_{\alpha n/2}$ satisfies the conclusion of the theorem. Indeed $T_{\alpha n/2}$ is a spanning rainbow tree by (i) and $d_{T_i}(u)\leq 3$ for $u\neq r$ by (ii).
Thus it remains to show that we can construct such a sequence of trees.

Let $T_0=D$ and notice that (i) -- (iv) hold by the assumptions of the lemma. 
For $0<i<\alpha n/2$, suppose that we have a tree $T_{i-1}$ satisfying (i) -- (iv). We will construct a tree $T_i$ satisfying (i) -- (iv). First we need the following claim, which identifies the vertices which need to be modified when passing from $T_{i-1}$ to $T_i$.
\begin{claim}\label{ClaimxyzwProperties}
There are four vertices  $x_i, y_i, z_i, w_i$ with the following properties.
\begin{enumerate}[(I)]
\item $x_i, y_i, z_i, w_i \in T_{i-1}$.
\item $x_i$ and $y_i$ are leaves of $T_{i-1}$, and $w_i$ is the (unique) neighbour of $x_i$ in $T_{i-1}$.
\item $x_iy_i\in E(G)$ with $c(x_iy_i)=c_i$
\item $d_{T_{i-1}}(z_i)\leq 2$.
\item $iz_i\in E(G)$ with $c(iz_i)=c(x_iw_i)$.
\item $z_i\neq x_i, w_i$. Also $x_i, y_i, w_i$ are distinct.
\end{enumerate}
\end{claim}
See Figure~\ref{Figure1FactorizationTree} to see what the vertices  $x_i, y_i, z_i, w_i$ look like.
\begin{figure}
  \centering
     \includegraphics[width=0.75\textwidth]{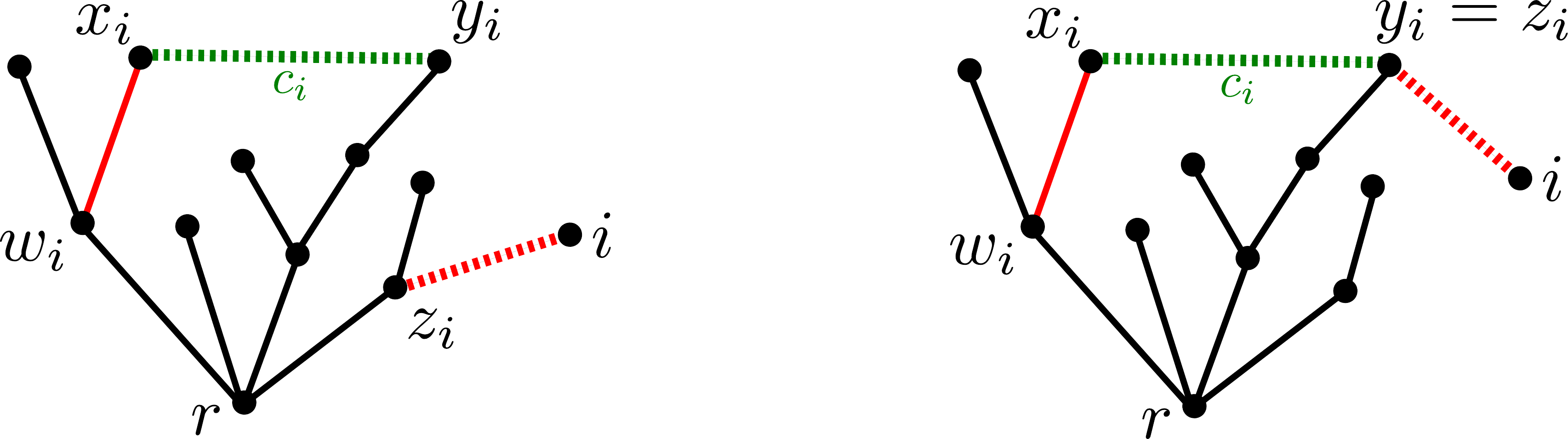}
  \caption{\small The vertices $i, x_i, y_i, z_i, w_i$ from Claim~\ref{ClaimxyzwProperties}. There are two slightly different cases pictured depending on whether $z_i= y_i$ or $z_i\neq y_i$.
  The solid edges are the edges of $T_{i-1}$ while the dashed edges are outside $T_{i-1}$. The tree $T_i=T_{i-1}+x_iy_i+iz_i-x_iw_i$ is constructed by switching the red solid edge for the red dashed edge, and also adding the green dashed edge.  \label{Figure1FactorizationTree}}
\end{figure}
\begin{proof}
Let $P$ be the set of colour $c_i$ edges in $G$ both of whose endpoints are leaves in $T_{i-1}$.
Since there are  at least $ (1-2\alpha)n/2$ colour $c_i$ edges in $G$,  at least $ (1-\alpha )n-i+1$ leaves in $T_{i-1}$ (by (iii)), and $i\leq \alpha n /2$, we have  $|P|\geq (1-2\alpha)n/2-\alpha n-i+1\geq (1-5\alpha)n/2$.

Since $d(i)\geq (1-\delta) n$, $T_{i-1}$ has at most $i$ vertices of degree  at least $ 3$ (by (ii) and (iv)), and $|T_{i-1}|= |D|+i-1= \left(1-\frac{\alpha}{2}\right)n+i-1$,  there is a set $Z\subseteq T_{i-1}\cap N(i)$ with  $|Z|\geq (1-\delta) n-\alpha n/2-1$ and $d_{T_{i-1}}(z)<3$ for all $z\in Z$. 

Since $G$ is properly coloured $P$ is a matching. For every vertex $x\in V(P)$, let $w_x$ be the neighbour of $x$ in $T_{i-1}$ ($w_x$ is unique since $x$ is a leaf in $T_{i-1}$). Since $T_{i-1}$ is rainbow (by (i)), we have that the colours $c(xw_x)$ are different for all $x\in V(P)$.  Since $|V(P)|+|Z|\geq  (1-5\alpha)n + (1-\delta-\alpha/2)n-1>n-1$, there is a colour $c$ which occurs in both $\{c(xw_x): x\in V(P)\}$ and $\{c(iz):z\in Z\}$. Let $x_i\in V(P)$ and $z_i\in Z$ be the vertices with $c(x_iw_{x_i})=c=c(iz_i)$. Let $w_i=w_{x_i}$ and let $y_i\in V(P)$ be the colour $c_i$ neighbour of $x_i$. 

We claim that (I) -- (VI) hold for $x_i, y_i, z_i, w_i$. 
For (VI), notice that we have ``$x_i\neq z_i$ and $w_i\neq z_i$''  since $c(x_iw_i)=c(iz_i)$ and $G$ is properly coloured. 
For (VI), we have that $x_i, y_i, w_i$ are distinct since $x_i$ and $y_i$ are two leaves of $T_{i-1}$ with $x_iy_i$ is an edge, and $w_i$ is not a leaf of $T_{i-1}$.
For condition (V), $iz_i\in E(G)$ comes from $Z\subseteq N(i)$ and $c(iz_i)=c(xw_i)$ comes from ``$c(x_iw_{x_i})=c=c(iz_i)$''.
Condition (IV) comes from ``$d_{T_{i-1}}(z)<3$ for all $z\in Z$'' and ``$z_i\in Z$''.
Condition (III) holds since $x_iy_i$ is an edge in $P$.
For condition (II), $x_i$ and $y_i$ being leaves of $T_{i-1}$ comes from all the vertices in $V(P)$ being leaves in $T_{i-1}$, while $w_i$ being the unique neighbour of $x_i$ in $T_{i-1}$ comes from $w_i=w_{x_i}$ and the definition of ``$w_x$''.
Condition (I) holds since $V(P),Z\subseteq T_{i-1}$ by definition of $V(P)$ and $Z$, and $w_i=w_{x_i}\in T_{i-1}$.
\end{proof}

Let $T_i=T_{i-1}+x_iy_i+iz_i-x_iw_i$. We claim that (i) -- (iv) hold for this tree.
First notice that the following all hold from Claim~\ref{ClaimxyzwProperties}, $T_{i-1}$ satisfying (i) -- (iv), and $T_i=T_{i-1}+x_iy_i+iz_i-x_iw_i$. 
\begin{align}\label{EqDegrees}
 \begin{array}{l}
\hspace{0.25cm} d_{T_i}(i)=1 \\
\hspace{0.05cm}d_{T_i}(x_i)=d_{T_{i-1}}(x_i)\\
d_{T_i}(w_i)=d_{T_{i-1}}(w_i)-1 \\
 \end{array}
 \hspace{0.2cm}
d_{T_i}(y_i)=
\begin{cases}
2&\text{ if $y_i\neq z_i$}\\
3&\text{ if $y_i= z_i$}
\end{cases}
\hspace{0.2cm}
d_{T_i}(z_i)=
\begin{cases}
d_{T_{i-1}}(z_i)+1&\text{ if $y_i\neq z_i$}\\
3&\text{ if $y_i= z_i$}
\end{cases}.
\end{align}
From ``$T_i=T_{i-1}+x_iy_i+iz_i-x_iw_i$'', we have that the only vertices whose degrees could change from $T_{i-1}$ to $T_i$ are  $i, x_i, y_i, z_i, w_i$.
For (iii), notice that $T_i$ has one new leaf (vertex $i$), and two vertices which were leaves in $T_{i-1}$ but may not be in leaves in $T_{i}$ (vertices $y_i$ and $z_i$). This shows that $T_i$ has at most one less leaf than $T_{i-1}$ which proves (iii).
For (ii), notice that (\ref{EqDegrees}) shows that all vertices, except $r$ and possibly $z_i$ have degree at most $3$ in $T_i$. We have $d_{T_{i-1}}(z_i)\leq 3$ by (\ref{EqDegrees}) and (IV).  For (iv),  notice that (\ref{EqDegrees}) shows that the only new vertex of degree $3$ in $T_i$ can be $z_i$.

For condition (i), notice that $T_i$ is  rainbow  using the colours of $T_{i-1}$ plus $c_i$ (since $T_i=T_{i-1}+x_iy_i+iz_i-x_iw_i$, the colour $c_i=c(x_iy_i)$ doesn't appear in $T_{i-1}$ and $c(iz_i)=c(x_iw_i)$).
We also have that $V(T_i)=V(T_{i-1})\cup \{i\}=V(D)\cup\{1, \dots, i\}$. Finally $T_{i}$  is a tree since it is obtained from the tree $T_{i-1}-x_iw_i$ by adding two leaves.
\end{proof}

Combining Lemmas~\ref{Lemma1FactorizationCoverSet} and~\ref{Lemma1FactorizationMakeSpanningTree}, it is easy to find $n/9-6$ edge-disjoint spanning rainbow trees in any properly $(n-1)$-coloured $K_n$ .
\begin{proof}[Proof of Theorem~\ref{Theorem1Factorization}]
Choose some $m \in [n/9-6, n/9-1]$ with $m\equiv 1$ or $3 \pmod 6$. Let $A$ be any set of $m$ vertices. 
By Lemma~\ref{Lemma1FactorizationCoverSet}, there is are edge-disjoint rainbow $(\leq n/18)$-spiders $D_1, \dots, D_{m}$  of order at least $(1-1/18)n$ with each $D_i$ rooted in $A$ and covering all the vertices in $A$. 

We  repeatedly apply Lemma~\ref{Lemma1FactorizationMakeSpanningTree} to the spiders $D_1, \dots, D_{m}$ in order to find disjoint spanning trees $T_1, \dots, T_{m}$ with $d_{T_i}(v)\leq 3$ for every $v\not\in A$. At the $i$th application, let $G=K_n\setminus (E(T_1)\cup\dots\cup E(T_{i-1})\cup E(D_{i+1})\cup\dots\cup E(D_{m}))$ and notice that we have $d_G(v)\geq (1-1/3)n$ for every $v\not\in A$. In addition, since the trees $T_1, \dots, T_{i-1}, D_{i+1},\dots, D_{m}$ are rainbow every colour has at least $n/2-n/9=(1-2/9)n/2$ edges in $G$. Therefore we can apply Lemma~\ref{Lemma1FactorizationMakeSpanningTree} with $\delta=1/3$ and $\alpha=1/9$ in order to find a spanning rainbow tree $T_i$ in $G$ with $d_{T_i}(v)\leq 3$ for every $v\not\in A$ as required. 
\end{proof}

\section{Isomorphic trees in proper colourings}\label{SectionIsomorphicTrees}
In this section we prove Theorem~\ref{TheoremProperColouring}. First we prove a number of auxiliary lemmas which we will need. 

\subsection*{Rainbow matchings}
To prove Theorem~\ref{TheoremProperColouring}, we will need some auxiliary results about rainbow matchings. We gather such results here. The following lemma gives an simple bound on how large a rainbow matching a coloured graph has. 
\begin{lemma}\label{LemmaLargeMatchingSmallMaximumDegree}
Let $G$ be a coloured graph with at most  $b$ edges of each colour. Then $G$ has a rainbow matching of size $\frac{e(G)}{2\Delta(G)+b}$.
\end{lemma}
\begin{proof}
Let $M$ be a maximum rainbow matching. Suppose for the sake of contradiction that $|M|<{e(G)}/(2\Delta(G)+b)$ The number of edges touching $V(M)$ is  at most  $\Delta(G)|V(M)|< {2\Delta(G) e(G)}/(2\Delta(G)+b)$. The number of edges sharing a colour with an edge of $M$ is  at most  $ b e(M)< {be(G)}/(2\Delta(G)+b)$. Since $e(G)={2\Delta(G) e(G)}/(2\Delta(G)+b)+{b e(G)}/(2\Delta(G)+b)$, there is an edge $e\in G$ which is disjoint from $V(M)$ and whose colour is not present in $M$. Thus $M\cup\{e\}$ is a rainbow matching, contradicting the maximality of $M$.
\end{proof}
We remark that the above lemma implies that every properly coloured graph has a rainbow matching of size $\frac{e(G)}{3|G|}$.
The above lemma is used to prove the following lemma about finding several disjoint rainbow matchings in a graph.
\begin{lemma}\label{LemmaManyLargeMatchings}
Let $G$ be a properly coloured graph with $\delta(G)\geq \delta$ and  at most  $b$ edges of each colour, and let $t\leq (|G|- 72\delta- 6b)/29$.
 Then $G$ has $t$ edge-disjoint rainbow matchings $M_1, \dots, M_t$ of size $\delta$. 

In addition there is a set $A=\{r_1, \dots, r_t\}$ with $A\cap V(M_i)=\emptyset$  such that for every $xy\in M_i$ we have one of  $r_ix\not\in E(G)$, $r_iy\not\in E(G)$, $c(r_ix)\not\in M_i$, or  $c(r_i y)\not\in M_i$.
\end{lemma}
\begin{proof}
The proof is by induction on $\delta$. The initial case when $\delta=0$ which holds trivially. 
Let $\delta>0$, and suppose that the lemma holds for all $\delta'<\delta$.
Let $b,t,G$ be as in the statement of the lemma.

Suppose there is a vertex $v\in V(G)$ with $d(v)\geq 6\delta+2t$. Notice  that $\delta(G\setminus \{v\})\geq \delta-1$.
Therefore, by induction $G\setminus\{v\}$ has $t$ edge-disjoint rainbow $(\delta-1)$-matchings $M_1, \dots, M_t$, and a set $A=\{r_1, \dots, r_t\}$ satisfying the conditions of the lemma. 
For $i=1, \dots, t$,  notice that out of the edges containing $v$, there are at most $t$ edges  touching $A$, at most $2\delta-2$ edges touching $V(M_i)$, at most $\delta-1$ edges sharing a colour with an edge of $M_i$, at most $\delta-1$ edges $vy$ with $c(r_iy)\in M_i$, and  at most $2\delta-2$ edges $vy$ with $c(vy)= c(r_i u)$ for $u\in V(M_i)$.
Therefore since $d(v)\geq 6\delta+2t$, for each $i$ there are at least $t$ edges $vy_i$  for which none of these occur.
Equivalently, for each $i$, there are at least $t$ edges $vy_i$  disjoint from $A$,   with $M_i\cup\{vy_i\}$ a rainbow matching, $c(r_iy_i)\not\in M_i$, and  $c(vy_i)\neq c(r_i u)$ for $u\in V(M_i)$. By greedily choosing such edges $vy_1, \dots, vy_t$ one at a time, we can ensure that they are all distinct, and and hence obtain disjoint rainbow matchings $M_1\cup\{vy_1\}, \dots, M_t\cup\{vy_t\}$ of size $\delta$  satisfying the conditions of the lemma.

Suppose that $\Delta(G)\leq 6\delta+2t$. 
Let $A$ be a set of $t$ vertices whose degrees in $G$ are as small as possible. By the choice of $A$, there is a number $d=\max_{r_i\in A}d(r_i)$ such that $d(v)\geq d\geq d(r_i)$ for all $r_i \in A$ and $v\not\in A$.  Let $H = G[V(G)\setminus A]$ to get a graph with $e(H)\geq d|H|/2-dt$.
By Lemma~\ref{LemmaLargeMatchingSmallMaximumDegree}, any subgraph $H'$ of $H$ with $e(H')\geq e(H)-t\delta$ has a rainbow matching $M$ satisfying  
$$e(M)\geq \frac{e(H')}{2\Delta(H')+b}
\geq \frac{0.5d|H|-td-t\delta}{2\Delta(G)+b}
\geq \frac{d(0.5|G|-2.5t)}{12\delta+4t+b}
\geq 3d\geq 3\delta(G)\geq 3\delta.$$
Here the third inequality comes from $|H|=|G|-t$, $\delta\leq \delta(G)\leq d$ and $\Delta(G)\leq 6\delta+2t$ while the fourth inequality is equivalent to $t\leq (|G|- 72\delta- 6b)/29$. 
For any $i$, given a rainbow matching $M=\{x_1y_2, \dots, x_{3\delta}y_{3\delta}\}$ of size $3\delta$, we can choose a submatching $M'\subseteq M$ of size $\delta$ such that we have either ``$r_ix_i\not\in E(G)$'' or ``$c(r_ix_i)\not\in M'$'' for any $x_i$ (to do this, choose the edges of $M'$ one at a time, noting that there are always less than $2|M'|$ edges of $M$ which can't be chosen).
By repeatedly choosing such matchings $M_1, \dots, M_\delta$ one at a time, at each step  letting $H'$ be $H$ minus the edges of the previously selected matchings, we get $t$ disjoint matchings of size $\delta$ as required. 
\end{proof}

\subsection*{Step 1: Disjoint spiders}
The following lemma allows us to find many disjoint nearly-spanning spiders in a graph. It is used as a starting point to finding the spanning spiders in Theorem~\ref{TheoremProperColouring}. This lemma is step (1) of the proof sketch in Section~\ref{SectionProofSketch}.
\begin{lemma}\label{LemmaCoverBipartiteGraph}
Let  $(1-2\delta)b\geq 8a$. Suppose that $K_{a,b}$ is properly coloured with bipartition classes $A$ and $B$ with $A=\{r_1, \dots, r_a\}$ and $|B|=b$.
Let $F_1, \dots, F_a$ be sets of colours with $|F_i|\leq \delta b$.

Then $K_{a,b}$ has edge-disjoint, rainbow $\left(a-1\right)$-spiders $S_1, \dots, S_{a}$, with $S_i$ rooted at $r_i$, $S_i$ having no colours from $F_i$, $|S_i|\geq (1-\delta)b-a+1$, and $V(S_i)\supset A$.
\end{lemma}
\begin{proof}
For every $i\neq j$ with $1\leq i,j\leq {a}$ we choose a vertex $b_{i,j}\in B$ such that $c(r_ib_{i,j}), c(r_j, b_{i,j})\not\in F_i$. 
Since there are always 
${b}-2|F_i|\geq (1-2\delta){b}\geq 8{a}$ choices for such a vertex,  we can ensure that for any $i,j,k,l$ with $\{i,j\}\cap \{k,l\}\neq \emptyset$ we have $c(r_ib_{i,j})\neq c(r_kb_{k,l})$ and $c(r_jb_{i,j})\neq c(r_kb_{k,l})$ (to see this, notice that for fixed $i,j$ there are less than $4{a}$ ordered pairs $(k,l)$ with $\{i,j\}\cap \{k,l\}\neq \emptyset$. Since there are  at least $ 8{a}$ choices for $b_{i,j}$ we can choose it so that $c(r_ib_{i,j}), c(r_jb_{i,j})$ are distinct from $c(r_kb_{k,l})$ for all $(k,l)$ with $\{i,j\}\cap \{k,l\}\neq \emptyset$).
Notice that since $K_{a,b}$ is properly coloured, this ensures that for distinct $i, j, k$, the vertices $b_{i,j}, b_{i,k}, b_{j,i}$, and $b_{k,j}$ are all distinct.

Let  $S_i^1=\{r_{i}b_{i,j}: j\neq i\}$, $S_i^2=\{r_jb_{i,j}: j\neq i\}$, and $S_i^3=\{r_i b: c(r_ib)\not\in F_i, \text{ $b \neq b_{i,j}$}$ and $\text{$b\neq b_{j,i}$ for any $j$}\}$ to get graphs with $e(S_i^1)=e(S_i^2)={a}-1$ and $e(S_i^3)\geq (1-\delta){b}-2{a}+2$. 
Notice that $S_i^1\cup S_1^3$ are rainbow since $K_{a,b}$ is properly coloured and $S^2_i$ is rainbow since $c(r_jb_{i,j})\neq c(r_kb_{i,k})$ for distinct $i,j,k$.
Notice that $S^1_i\cup S^2_i$ is rainbow  since $c(r_ib_{i,j})\neq c(r_kb_{i,k})$ for distinct $i,j,k$.
Since $S_i^3$ is rainbow and $|S_i^3|\geq (1-\delta){b}-2{a}+2\geq {a}-1$, we can delete some set of $|S^2_i|={a}-1$ edges from $S_i^3$ to get a set $\hat S_i^3$ such that $\hat S_i^3\cup S_i^2$ is rainbow. 

For each $i=1, \dots, {a}$, let $S_i=S^1_i\cup S^2_i\cup \hat S^3_i$ to get a rainbow $({a}-1)$-spider of size  at least $(1-\delta){b}-{a}+1$. Notice that $S_i$ covers $A$ since $S_i^2$ covers $A\setminus\{r_i\}$. For distinct $i$ and $j$, $S_i$ and $S_j$ are edge-disjoint since for distinct $i, j, k$, the vertices $b_{i,j}, b_{i,k}, b_{j,i}$, and $b_{k,j}$ are all distinct.
\end{proof}

\subsection*{Step 2: Spanning spiders}
The above lemma finds many disjoint nearly-spanning spiders in a graph. In order to prove Theorem~\ref{TheoremProperColouring}, we need to turn these into truly spanning spiders i.e. we need to perform step (2) of the proof sketch from Section~\ref{SectionProofSketch}. 

The following lemma is used to do this---it says that under certain conditions, a rainbow star can be extended to a rainbow spider covering one extra vertex.
\begin{lemma}\label{LemmaAddOneVertexToStar}
Let   $\delta$ and $\mu$ be in $(0,1)$ with $2\mu|G|> 2\delta|G|+5$ and $1-\delta>4\mu$.
Let $G$ be a properly coloured graph with $\delta(G)\geq (1-\delta)|G|$,
 $S$ a star in $G$ rooted at $r$ with $|S|=|G|-1$,  and $M$ a matching in $G$ with $\mu |G|$ edges sharing no colours with $S$. 
Then $G$ has a spanning rainbow $(\leq 3)$-spider $D$ rooted at~$r$.
\end{lemma}
\begin{figure}
  \centering
     \includegraphics[width=1\textwidth]{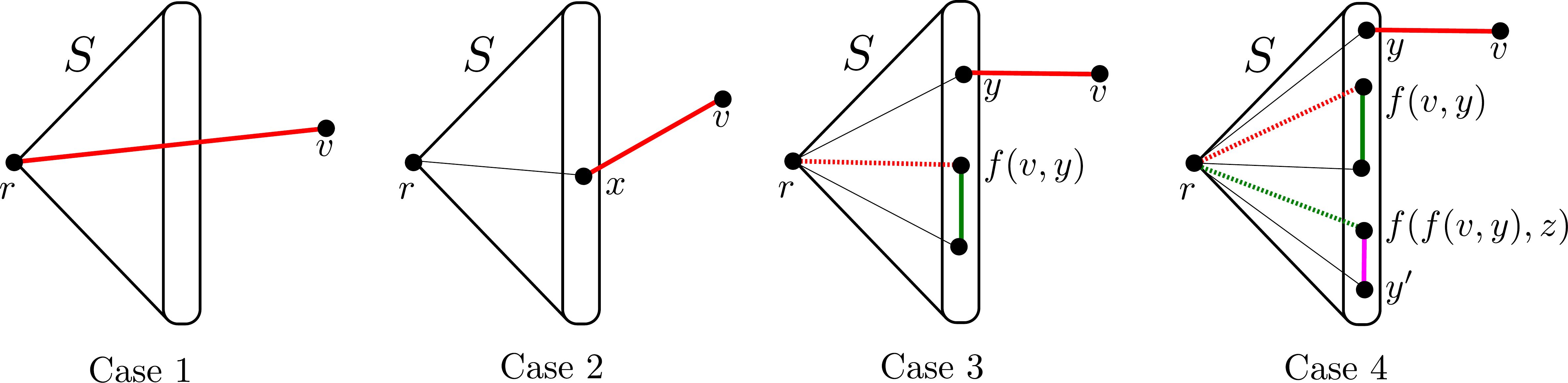}
  \caption{The different cases of the proof of Lemma~\ref{LemmaAddOneVertexToStar}. Each of the figures represent a different way of modifying the star $S$ to produce a spider  $D$ containing $v$. Dashed edges represent edges which get deleted from $S$ to get $D$, while dashed edges represent edges which get added to $S$ to get $D$.  \label{FigureAddOneVertexToStar}}
\end{figure}
\begin{proof}
Let $v$ be the vertex in  $V(G)\setminus V(S)$.
If $rv$ is an edge then $S+rv$ is a rainbow $0$-spider satisfying the conclusion of the lemma (see Case 1 in Figure~\ref{FigureAddOneVertexToStar}). 
If $c(vx)\not\in S$ for any vertex $x\in N(v)\setminus\{r\}$, then $S+vx$ is a rainbow $1$-spider satisfying the conclusion of the lemma (see Case 2 in Figure~\ref{FigureAddOneVertexToStar}). 
Therefore, we can assume that $c(vx)\in S$ for every $x\in N(v)$. In particular we have that $v, r\not\in V(M)$.

For a vertex $x$ let $N_S(x)=\{y\in N(x):  c(xy)\in S\}$. From the previous paragraph, we have $N_S(v)=N(v)$.
For $x\in V(G)$ and  $y\in N_S(x)$, let $f(x,y)$ be the vertex  $s\in S$ with $c(rs)=c(xy)$. 
Since $G$ is properly coloured, for fixed $x$ the function $f(x,y)$ is an injection from  $N_S(x)$ to $V(S)$. Notice that since $G$ is properly coloured and $c(r f(x,y))=c(xy)$, we  have 
\begin{equation}\label{Eq_y_and_x_distinct_from_fxy}
y\neq f(x,y) \text{ and } x\neq f(x,y) \text{ for any } x,y\neq r.
\end{equation}

Suppose that there is some  $y\in N_S(v)=N(v)$ and $z\neq y$ with $zf(v,y)\in M$. We claim that the edges $vy$ and $zf(v,y)$ are disjoint. Indeed $y\neq z$ by assumption, $y\neq f(v,y)$ by (\ref{Eq_y_and_x_distinct_from_fxy}), and $v\cap \{z,f(v,y)\} =\emptyset$ since $v\not\in V(M)$.
Using Observation~\ref{ObservationStarPlusMatching} we have that $D=S-rf(v,y)+vy+zf(v,y)$ is a rainbow $2$-spider satisfying the conclusion of the lemma (see Case 3 in Figure~\ref{FigureAddOneVertexToStar}). 
Therefore, for the rest of the proof we can assume the following.
\begin{equation}\label{Eq_yfxy_in_M}
\text{For $y\in N(v)$ with $f(v,y)\in V(M)$ we have $yf(v,y)\in M$.}
\end{equation} 

Suppose that there is  $y\in N_S(v)$ and $z\in N(f(v,y))$ with $z\not\in\{y,v,r,f(v,y)\}$ and $c(zf(v,y))\not\in S$. Notice that $y,$ $v,$ $r,$ $f(v,y)$ are all distinct.
Using Observation~\ref{ObservationStarPlusMatching}, $S-rf(v,y)+vy+zf(v,y)$ is a rainbow $2$-spider satisfying the conclusion of the lemma (see Case 3 in Figure~\ref{FigureAddOneVertexToStar}).  Therefore we can assume that for all $y\in N_S(v)=N(v)$ we have $N(f(v,y))\setminus N_S(f(v,y))\subseteq \{y,v,r,f(v,y)\}$. 
Together with $f(v,y) \not\in  N(f(v,y))$, $r\in N_S(f(v,y))$, and  $\delta(G)\geq (1-\delta)|G|$  this implies 
\begin{equation}\label{Eq_NSfvy_bound}
|N_S(f(v,y))|\geq (1-\delta)|G|-2.
\end{equation}

Since $|N(v)|\geq (1-\delta)|G|>4\mu |G|\geq 2|V(M)|$ and $f(x,y)$ is an injection for fixed $x$, there is some $y\in N(v)=N_S(v)$ with $y\not\in V(M)$ and $f(v,y)\not\in V(M)$. Let $T=\{f(v,y'):y'\in N(v)\}$ and notice that $|T|=|N(v)|\geq(1-\delta)|G|$ and so $|T\cap V(M)|\geq (2\mu-\delta)|G|$.
Using (\ref{Eq_NSfvy_bound}) and ``$2\mu|G|> 2\delta|G|+5$'' we have 
$$|N_S(f(v,y))|\geq (1-\delta)|G|-2> (1-2\mu+\delta)|G|+3\geq |V(G)\setminus (T\cap V(M))|+|\{v,r,y\}|.$$
Using the fact that $f(x,y)$ is an injection for fixed $x$,  there is some $z\in N_S(f(v,y))\setminus\{v,r,y\}$ with $f(f(v,y),z)\in V(M)\cap T$. 
Since $f(f(v,y),z)\in T$, there is some $y'\in N(v)$ with $f(f(v,y),z)=f(v,y')$. Since  $f(f(v,y),z)=f(v,y')\in V(M)$ from (\ref{Eq_yfxy_in_M}) we get that $f(f(v,y),z)y'\in M$. By the definition of $f(*,*)$, we have $c(f(f(v,y),z)r)=c(f(v,y)z)$ and $c(f(v,y')r)=c(vy')$. Since $f(f(v,y),z)=f(v,y')$, we get $c(vy')=c(f(v,y)z)$. Since $G$ is properly coloured, we get that $z\neq y'$ and $f(v,y)\neq y'$.

Notice that the edges $vy$, $f(v,y)z$, and $f(f(v,y),z)y'$ are disjoint. 
Indeed we have and $z\neq y'$ and $f(v,y)\neq y'$ from the previous paragraph. 
We have $v\not\in\{ f(v,y),z,f(f(v,y),z),y'\}$ since $v\not\in  V(S)\cup V(M)$ and by choice of $z$. 
We have   $y\neq f(v,y),z$ by (\ref{Eq_y_and_x_distinct_from_fxy}) and choice of $z$.
We have $y\neq f(f(v,y),z),y'$ since $y\not\in V(M)$. 
Finally, we have  $f(f(v,y),z)\neq f(v,y), z$ by (\ref{Eq_y_and_x_distinct_from_fxy}). 

Notice that the edges $vy$, $f(v,y)z$, and $f(f(v,y),z)y'$ have different colours. Indeed we have $c(vy)\neq c(f(v,y)z)$ and  $c(f(v,y)z)\neq c(f(f(v,y),z)y')$ since $c(ab)=c(r f(a,b))$ for any edge $ab$ and $G$ is properly coloured. We have  $c(vy)\neq c(f(f(v,y),z)y')$ since $f(f(v,y),z)y'\in M$ and $c(vy)\in S$.

Now we have that the edges $vy$, $f(v,y)z$, and $f(f(v,y),z)y'$ are disjoint and have different colours. Using $c(vy)=c(f(v,y)r)$ and $c(f(v,y)z)=c(f(f(v,y),z)r)$ and Observation~\ref{ObservationStarPlusMatching} we have that $S-f(v,y)r-f(f(v,y),z)r+vy+f(v,y)z+f(f(v,y),z)y'$ is a rainbow $3$-spider (see Case 4 in Figure~\ref{FigureAddOneVertexToStar}). 
\end{proof}

By iterating the above lemma, we can show that under certain conditions, if we have a sufficiently large star, then we also have a spanning spider.
\begin{lemma}\label{LemmaAddManyVerticesToStar}
Let $\epsilon, \phi, \delta, \tau>0$.
 Let $G$ be a sufficently large properly coloured graph and set  $t=\tau|G|$. 
Suppose that $\delta(G)\geq (1-\delta)|G|+2t$ and $S$ is a  star centered at $r\in V(G)$ with $|S|=|G|-t$.  Suppose that either of the following hold.
\begin{enumerate}[(i)]
\item There are   at least  $t$ colours outside $S$, each with   at least   at least $ \epsilon |G|+t$ edges and $\epsilon\geq \delta+ 19\tau$.
\item There are   at most  $(1-\phi)(|G|-t)$ colours in $G$, each with   at least  $ \epsilon(|G|-t)$ edges and $0.1\geq \epsilon\geq \phi\geq 13\delta+200\tau$.
\end{enumerate}
Then $G$ has a spanning rainbow $(\leq 3t)$-spider centered at $r$.
\end{lemma}
\begin{proof}
The proof is by induction on $t$. The initial case ``$t=0$'' is trivial since $S$ is a $(\leq 0)$-spider centered at $r$. Suppose that $t\geq 1$ and the lemma holds for all $t'<t$.
Let $v$ be a vertex  not in $S$. Let $c^+$ be a colour outside $S$ with a maximum number of edges. 
Let $H$ be the subgraph of $G$ on $V(G)\setminus\{v\}$ with colour $c^+$ edges deleted. 

We show that the assumptions of the lemma hold for the graph $H$ and star $S$ with $t'=t-1$.
We have $\delta(H)\geq \delta(G)-2\geq (1-\delta)|G|+2t-2$, and $S$ is a  star in $H$ with $|S|=|H|-t+1$. 
If (i) held for $G$, then $H$ has  at least  $ t-1$ colours outside $S$, each with   at least $ \epsilon |G|+t-1\geq \epsilon|H|+t-1$ edges.
If (ii) held for $G$ then $H$ has  at most  $ (1-\phi)(|G|-t)=(1-\phi)(|H|-t+1)$ colours with  at least $ \epsilon(|G|-t)=\epsilon(|H|-t+1)$ edges. 

By induction $H$ has a spanning $(\leq 3t-3)$-spider $D$ rooted at $r$. 
Let $D=S'\cup D'$ where $S'$ is a star with $|S'|\geq |D'|-6(t-1)$ consisting of the ordinary leaves of $D$, and $D'$ is a $(\leq 3t-3)$-spider with $|D'|\leq 6(t-1)+1$.

Let $G'$ be the subgraph of $G$ on $V(S')\cup \{v\}$ consisting of all colours not in $D'$, $\delta'=\delta+12\tau$ and $\mu=\delta+13\tau$.
We show that the requirements of Lemma~\ref{LemmaAddOneVertexToStar} hold for $G'$, $\delta'$ and $\mu$.
Since   at most  $6t$ colours and  at most  $6t$ vertices are missing from $G'$, we have $\delta(G')\geq \delta(G)-12t\geq (1-\delta-12\tau)|G|=  (1-\delta')|G|\geq (1-\delta')|G'|$.
We have that $S'$ is a  star in $G'$ rooted at $r$ with  $V(G')=V(S')\dot\cup\{v\}$. 

If (i) holds for $G$, then since $c^+\not\in D'$ the colour $c^+$ edges in $G'$ form a matching $M$ of size  at least $ \epsilon |G|-6t\geq \mu|G'|$ disjoint from $S'$.

If (ii) holds for $G$, then notice that the number of edges in $G'$ of colours on $S'$ is at most 
\begin{align*}
(1-\phi)(|G|-t)\frac{|G|}{2}+ (|S'|- (1-\phi)(|G|-t))(\epsilon|G|-t)
&\leq (1-\phi)(1-\tau)\frac{|G|^2}2\\
&\hspace{2cm}+ (\phi+\tau(1-\phi))(\epsilon-\tau)|G|^2\\
&\leq \left(\frac12 -\frac{\phi}2+ \phi\epsilon +\tau \right)|G|^2\\
&\leq \frac{|G'|^2}2- \left(\frac25\phi -14\tau \right)|G|^2.
\end{align*}
On the LHS, the ``$(1-\phi)(|G|-t)\frac{|G|}{2}$'' term comes from the  colours in $G$ with   at least $ \epsilon(|G|-t)$ edges (of which there are at most $(1-\phi)(|G|-t)$), and the ``$(|S'|- (1-\phi)(|G|-t))(\epsilon|G|-t)$'' term comes from the other colours in $S$ having less than $(\epsilon|G|-t)$ edges. 
The first inequality comes from $|S'|\leq |G|$, $t=\tau|G|$, and $\epsilon>\tau$. The second inequality comes from $\tau\leq\phi\leq\epsilon\leq0.1$ and rearranging. The third inequality comes from $|G'|\geq |G|-6t$ and $\epsilon\leq0.1$. 
Thus the number of  edges in $G'$ of colours outside $S'$ is at least 
\begin{align*}
e(G')-\frac{|G'|^2}2+ \left(\frac25\phi -14\tau\right)|G|^2
&\geq (1-\delta')\frac{|G'|^2}2-\frac{|G'|^2}2+ \left(\frac25\phi -14\tau\right)|G|^2\\
&\geq \left(\frac25\phi -14\tau-2\delta'\right)|G|^2.
\end{align*}
 Lemma~\ref{LemmaLargeMatchingSmallMaximumDegree}  and ``$\epsilon\geq \phi\geq 13\delta+200\tau$'' give a rainbow matching $M$ of size $(2\phi/5 -14\tau-2\delta')|G|^2/3|G|\geq \mu|G|$ using only colours outside $S'$.
 
In either of the above cases, we obtained a matching of size  at least $ \mu|G|$ in $G'$ consisting of colours outside $S'$.
Since $2\mu|G|\geq 2\delta'|G|+5$ and $1-\delta'\geq 4\mu$, we can apply Lemma~\ref{LemmaAddOneVertexToStar} to get a rainbow $(\leq 3)$-spider $D''$ in $G'$ rooted at $r$. 
Since $G'$ and $D'$ share no colours, Observation~\ref{ObservationTwoDoubleStars} shows that $D'\cup D''$ is a rainbow $(\leq 3t)$-spider rooted at $r$ as required.
\end{proof}

By interating the above lemma it is possible to find many edge-disjoint spanning spiders. 
\begin{lemma}\label{LemmaDisjointSpanningDoubleStars}
Let $\epsilon, \phi, \alpha, \gamma, \tau>0$ and $n$ be sufficiently large.
Let $K_n$ be properly coloured, and $D_1, \dots, D_{\alpha n}$  edge-disjoint rainbow $(\leq \gamma n)$-spiders in $K_n$ with $D_i$ a rainbow $t_i$-spider rooted at $r_i$ for each $i$ satisfying  $|D_i|\geq (1-\tau)n$.
Suppose that $r_i\in D_j$ for all $i,j$, and one of the following holds.
\begin{enumerate}[(i)]
\item For each $i$, there are    at least $ n-|D_i|$ colours outside $D_i$, each with   at least $ \epsilon n$ edges and $\epsilon\geq  9\alpha+8\gamma+ 25\tau$.
\item There are  at most $(1-\phi)n$ colours in $K_n$, each with   at least $ \epsilon n$ edges and  $0.03\geq \epsilon/2-0.001\geq \phi\geq 80\alpha+50\gamma+ 340\tau$.
\end{enumerate}
Then $K_n$ has $\alpha n$ edge-disjoint spanning rainbow $(\leq (\gamma +3\tau)n)$-spiders $D'_1, \dots, D'_{\alpha n}$ with $D'_i$ rooted at $r_i$.
\end{lemma}
\begin{proof}
For each $i$, let $D_i=S_i\cup \hat D_i$ where $S_i$ is the star consisting of the ordinary leaves of $D_i$ disjoint from $\{r_1, \dots, r_{i-1}, r_{i+1}, \dots, r_{\alpha n}\}$. Notice that we have $|\hat D_i|\leq (2\gamma+\alpha)n$.
For $i=1, \dots, \alpha n$  we will apply  Lemma~\ref{LemmaAddManyVerticesToStar} to  $S_i$ with $\delta=5\alpha+4\gamma+2\tau$, $t=n-|D_i|$, and appropriate $\epsilon'$, $\phi'$, and $G$ in order to get a $(\leq 3\tau n)$-spider $D'_i$ rooted at $r_i$.

At the $i$th application, let $G$ be the subgraph of $K_n$ on $(V(K_n)\setminus V(\hat D_i))\cup \{r_i\}$ consisting of all the edges which are not in $D'_1, \dots, D'_{i-1}, D_{i+1}, \dots, D_{\alpha n}$, and which don't share any colours with $\hat D_i$.
Notice that since the induced subgraphs of $D_j$ and $D'_j$ on $V(G)$ have maximum degree $1$, we have $\delta(G)\geq n-\alpha n-|\hat D_i|-e(\hat D_i)\geq (1-3\alpha -4\gamma )n\geq  (1-\delta+ 2\tau)n\geq (1-\delta)|G|+2\tau |G|$ (using $|\hat D_i|\leq (2\gamma+\alpha)n$ and $\delta=5\alpha+4\gamma+2\tau$). We also have $|G|\geq (1-2\gamma-\alpha)n$.
We claim that either part (i) or (ii) of Lemma~\ref{LemmaAddManyVerticesToStar} holds for $G$.

If we are in case (i), let $\epsilon'= \epsilon-4\alpha-2\gamma-\tau$ and notice that  $\epsilon'\geq \delta+ 19\tau$. holds.
Notice that we have  at least $ n-|D_i|=|G|-|S_i|$ colours in $G$ outside $S_i$ each with  at least $ \epsilon n -2\alpha n -2|\hat D_{i}|\geq (\epsilon-4\alpha-2\gamma)n=(\epsilon'+\tau) n\geq (\epsilon'+\tau) |G|$ edges. This shows that part (i) of Lemma~\ref{LemmaAddManyVerticesToStar} holds.

If we are in case (ii), let $\epsilon'=\epsilon/(1-2\gamma)+\tau$,  $\phi'=1-(1-\phi)/(1-\tau)(1-2\gamma-\alpha)$, and notice that  $0.1\geq \epsilon'\geq \phi'\geq 13\delta+200\tau$ holds. 
Notice that we have  at most $(1-\phi)n\leq (1-\phi')(1-\tau)|G|$ colours in $G$  with  at least $   (\epsilon'-\tau) |G|\geq \epsilon n$ edges. This shows that part (ii) of Lemma~\ref{LemmaAddManyVerticesToStar} holds.

Since all the assumptions of Lemma~\ref{LemmaAddManyVerticesToStar} hold for $G$, we can apply it to get a spanning $(\leq 3\tau n)$-spider $\tilde D_i$ in $G$. By Observation~\ref{ObservationTwoDoubleStars}, $D_i'=\tilde D_i \cup \hat D_i$ is a  $(\leq (\gamma +3\tau)n)$-spider rooted at $r_i$ as required. 
\end{proof}

\subsection*{Step 3: Isomorphic spiders}
In Theorem~\ref{TheoremProperColouring} we want to find many spanning isomorphic spiders. In the proof it is more convenient to first find many spanning non-isomorphic spiders, and later modify them to isomorphic ones. In this section we prove a result about changing $t$-spiders into $s$-spiders for $s>t$. The results in this section are the essence of step (3) in the proof sketch in Section~\ref{SectionProofSketch}.

A total colouring of a directed graph $D$ is an assignment of colours to all the edges and vertices of $D$.
We say that a totally coloured directed graph $D$ is properly coloured if for any vertex $v$ we have $c(xv)\neq c(yv)$, $c(vx)\neq c(vy)$, $c(xv)\neq c(v)$, and $c(vx)\neq c(v)$  for distinct $x,y,v$. Notice that we do not forbid $c(xv)=c(vy)$. A totally coloured graph is vertex-rainbow if all its vertices have different colours. We'll need the following lemma.
\begin{lemma}\label{LemmaRainbowCycle}
Let $D$ be a properly totally coloured, vertex-rainbow directed graph with $e(D)\geq (1-\delta)|D|^2$.
Then $D$ has a rainbow cycle of length $s$ for any $3\leq  s< \frac{1-9\sqrt{\delta}}{12}|D|$.
\end{lemma}
\begin{proof}
Let $D'$ be the induced subgraph of $D$ consisting of vertices $v$ with $|N^+(v)|\geq (1-\sqrt{\delta}) |D|$ and $|N^-(v)|\geq (1-\sqrt{\delta}) |D|$. 
Since $e(D)\geq (1-\delta)|D|^2$, there are at most $\sqrt{\delta}|D|$ vertices in $D$ with $|N^-(v)|< (1-\sqrt{\delta}) |D|$ and at most $\sqrt{\delta}|D|$ vertices with $|N^+(v)|< (1-\sqrt{\delta}) |D|$.
These imply that  $\delta^-(D'), \delta^+(D')\geq (1-3\sqrt{\delta}) |D|$. 

Choose a sequence of vertices $v_1, v_2, \dots, v_{s-2}\in V(D')$  with $v_{i+1}$ chosen from $N^+_{D'}(v_i)$ with $c(v_iv_{i+1}), c(v_{i+1})\not\in \{c(v_1), \dots, c(v_{i})\}\cup \{c(v_1v_2), \dots, c(v_{i-1}v_i)\}$. This is possible since $D$ is properly coloured, vertex-rainbow,  $s\leq \frac{1-9\sqrt{\delta}}{11}|D|$, and $|N^+_{D'}(v_i)|\geq (1-3\sqrt{\delta}) |D|$. We have that $P=v_1, v_2, \dots, v_{s-2}$  is a rainbow path.

Notice that out of the edges $xy$ with $x\in N^+_{D'}(v_{s-2})$ and $y\in N^-_{D'}(v_1)$ there are at most $10s|D|$ edges with $c(v_{s-2}x)$, $c(x)$, $c(xy)$, $c(y)$, or $c(yv_1)$ occuring in $P$, and at most $5|D|$ edges $xy$ for which any of $c(v_{s-2}x)=c(xy)$, $c(v_{s-2}x)=c(y)$, $c(x)=c(yv_1)$, $c({x}y)=c(yv_1)$, or $c(v_{s-2}x)=c(yv_1)$ hold. Since there are at least 
$(\delta^+(D')-|D'\setminus N^-_{D'}(v_{1})|)|N^+_{D'}(v_{s-2})|\geq (1-6\sqrt{\delta})|D|(1-3\sqrt{\delta})|D|\geq (1-9\sqrt{\delta})|D|^2>12s|D|$ edges from $N^+_{D'}(v_{s-2})$ to $N^-_{D'}(v_1)$, there must be at least one edge $xy$ for which none of these occur. Now $v_1, v_2, \dots, v_{s-2}, x, y$ is a rainbow cycle of length $s$ as required.
\end{proof}

The following lemma allows us to increase the parameter in a spider.
\begin{lemma}\label{LemmaChangeDoubleStarParameter}
Let $G$ be a sufficiently large properly coloured graph with $|N(v)|\geq (1-\delta)|G|$ holding for at least $(1-\delta)|G|$ vertices in $G$.
For $t\leq \delta|G|$, let $D_0$ be a spanning rainbow $t$-spider in $G$ which is rooted at $r$. 
Then for any $s$ with $3\leq s\leq (0.001-8\delta)|G|$, $G$ has a spanning rainbow $(t+s)$-spider rooted at $r$.
\end{lemma}
\begin{proof}
Let $r$ be the root of $D_0$. Let $D_0=S\cup \hat D$ where $S$ is the star consisting of the ordinary leaves of $D_0$ and $\hat D$ is a $t$-spider. Let $B$ be the set  of  at most $\delta|G|$ vertices of degree less than $(1-\delta)|G|$ in $G$.
Let $H$ be the subgraph of $G$ on $V(S)\setminus B$ consisting of the colours not in $\hat D$.  We have 
$\delta (H)\geq \delta(G)-e(\hat D)- |V(G)\setminus V(H)|=
(1-\delta)|G|-e(\hat D)-v(\hat D)+1-|B|\geq (1-6\delta )|G|$.
Using Observation~\ref{ObservationTwoDoubleStars}, to prove the lemma it is sufficient to find a spanning rainbow $s$-spider in $H$ which is rooted at $r$.
Let $M$ be a maximum rainbow matching in $H$  consisting of colours not on $S$. 

Suppose that $e(M)\geq  (0.001-8\delta)|G|$. Let $M'=\{x_1y_1, \dots, x_{s}y_s\}$ be a submatching of $M$. Since $M$ doesn't share colours with $S$, Observation~\ref{ObservationStarPlusMatching} shows that $D'=S\cup M'\setminus\{rx_1, \dots, rx_{s}\}$ is a spanning rainbow $s$-spider in $H$ as required.

Suppose that $e(M)\leq  (0.001-8\delta)|G|$. Let $J$ be the subgraph of $H$ on $H\setminus (V(M)\cup \{r\})$ consisting of colours  not on $M$. 
We have $\delta(J)\geq \delta(H)-3e(M)-1\geq  (0.997+18\delta)|G|-1\geq 0.995|J|$ and $e(J)\geq \delta(J)|J|/2\geq 0.997|J|^2/2$. By maximality of $M$, all colours on $J$ occur in $S$.
We construct an auxiliary totally coloured digraph $D$ with vertex set $V(J)$ whose set of colours is also $V(J)$. 
For $x,y,z\in V(J)$  we let $xy$ be a colour $z$ edge in $D$ whenever there is a colour $c(rx)$ edge between $z$ and $y$ in $G$. If there is no colour $c(rx)$ edge touching $y$ in $G$, then there is no edge $xy$ in $D$.
We colour every vertex $v$ by itself.
Notice that every edge in $J$ contributes exactly twice to $D$, giving $e(D)=2e(J)\geq 0.995|D|^2$.
Notice that $D$ is properly coloured with rainbow vertex set. Indeed $vx$ and $vy$ cannot have the same colour because $G$ is properly coloured, $xv$ and $yv$ cannot have the same colour since $G$ is simple, $vx$ is not coloured by $v$ since $G$ is properly coloured, and $xv$ is not coloured by $v$ since $G$ is loopless.

By Lemma~\ref{LemmaRainbowCycle} applied with $\delta=0.001$, $D$ has a rainbow cycle $C=x_1x_2, \dots, x_s$ of length $s\leq \frac{1-9\sqrt{0.001}}{12}|D|$. Let $m_i$ be the edge of $G$ corresponding to $x_ix_{i+1 \pmod s}$ i.e. let if $c(x_ix_{i+1 \pmod s})=z$, then we let $m_i=zx_{i+1 \pmod s} \in E(G)$. 
Let $M'=\{m_1, \dots, m_s\}$. Notice that $M'$ is a matching since $C$ is rainbow and has in-degree $1$. Notice that $M'$ is rainbow since $C$ has out-degree $1$. 
By definition of $D$ we know that $x_i\in m_{i-1\pmod s}$ and $m_i$ has the same colour as $rx_i$.
By Observation~\ref{ObservationStarPlusMatching}, $D'=S\cup M'\setminus\{rx: x \in V(C)\}$ is a $(t+s)$-spider satisfying the lemma.
\end{proof}

\subsection*{Proof of Theorem~\ref{TheoremProperColouring}}
We now prove the main result of this section.
\begin{proof}[Proof of Theorem~\ref{TheoremProperColouring}]
In this proof let $\alpha=0.000001$, $\phi=0.0005$, and $\epsilon=0.06$.
Let $C_F$ be the set of  colours  which each have  at least $ \epsilon n$ edges. Notice that  one of the following holds.
\begin{enumerate}[(a)]
\item $\phi n \geq n-|C_F|$,
\item $|C_F|\leq (1-\phi) n$.
\end{enumerate}
Our proof will be slightly different depending on which of the above cases occurs. 

First we define a set of vertices $A=\{r_1, \dots, r_{\alpha n}\}$ of size $\alpha n$. If we are in case (b), let $A$ be  an arbitrary set of this size.
If we are in case (a), first let $H$ be the subgraph of $K_{n}$ of colours not it $C_F$. Notice that $\delta (H)\geq  n-|C_F|-1$, every colour in $H$ occurs at most $\epsilon n$ times,  and $\alpha n\leq (n-72\phi n- 6\epsilon n)/29\leq (n-72(n-|C_F|-1)- 6\epsilon n)/29$. By Lemma~\ref{LemmaManyLargeMatchings} applied with $G=H$, $\delta=n-|C_F|-1$, $b=\epsilon n$, and $t=\alpha n$ we can choose rainbow matchings $M_1, \dots, M_{\alpha n}$ of size $(n-|C_F|-1)$ and a set $A=\{r_1, \dots, r_{\alpha n}\}$ of size $\alpha n$ disjoint from $M_1, \dots, M_{\alpha n}$. In addition  for every $xy\in M_i$ either $c(r_ix)\not\in M_i$ or $c(r_i y)\not\in M_i$.

Next, we let $B=V(K_n)\setminus A$ and apply Lemma~\ref{LemmaCoverBipartiteGraph} to the complete bipartite graph $K_n[A,B]$.
If we are in case (b), we do this with $F_1, \dots, F_{\alpha n}=\emptyset$ and $\delta=0$. Using $|B|=(1-\alpha)n\geq 8\alpha n = 8|A|$, Lemma~\ref{LemmaCoverBipartiteGraph}  gives us   $\alpha n$ edge-disjoint rainbow $(\alpha n-1)$-spiders $D_1, \dots, D_{\alpha n}$ with $D_i$ rooted at $r_i\in A$, $D_i$ covering $A$, and $|D_i|\geq |B|-|A|+1\geq (1-2\alpha)n$.
If we are in case (a), we apply Lemma~\ref{LemmaCoverBipartiteGraph}  with  $F_i=\bigcup_{xy\in M_i}\{c(xy), c(r_ix), c(r_iy)\}$ and $\delta=4\phi$, which satisfy  $|F_i|\leq 3\phi n\leq 4\phi|B|$ and $(1-4\phi)|B|=(1-4\phi)(1-\alpha)n\geq3\alpha n = 3|A|$.
  Lemma~\ref{LemmaCoverBipartiteGraph}  gives us $\alpha n$ edge-disjoint rainbow $(\alpha n-1)$-spiders $D_1, \dots, D_{\alpha n}$ with $D_i$ rooted at $r_i\in A$, $D_i$ covering $A$, $D_i$ having no colours from $F_i$, and $|D_i|\geq (1-4\phi)|B|-|A|+1\geq(1-4\phi-2\alpha)n$. 
  Notice that since $c(r_ix), c(r_iy)\in F_i$ we have that $D_i$ is vertex-disjoint from $M_i$ (using the fact that all vertices in $D_i\cap B$ are neighbors of $r_i$ since $D_i$ is a spider contained in $K_n[A,B]$).

Next we apply Lemma~\ref{LemmaDisjointSpanningDoubleStars} to $K_n$ in order to get edge-disjoint spanning rainbow $(\leq 0.0002 n)$-spiders $D'_1, \dots, D'_{\alpha n}$.
If we are in case (b), notice that part (ii) of Lemma~\ref{LemmaDisjointSpanningDoubleStars} holds with $\alpha=\alpha$,  $\gamma=\alpha$, $\tau=2\alpha$, $\phi=\phi$, and $\epsilon=\epsilon$. Therefore we can apply Lemma~\ref{LemmaDisjointSpanningDoubleStars} to get the required spiders.
If we are in case (a), recall that by construction of $M_i$ and $A$ in Lemma~\ref{LemmaManyLargeMatchings}, for each $m\in M_i$ there is a vertex $x_m\in m$ with $c(r_ix_m)\not\in M_i$.
Let $D^1_i=D_i\cup M_i\cup\{r_ix_m: m\in M_i\}$. Notice that $D_i\cup\{r_ix_m: m\in M_i\}$  and  $D_i\cup M_i$ are rainbow by choice of $F_i$ in our application of Lemma~\ref{LemmaCoverBipartiteGraph}, and $M_i\cup\{r_ix_m: m\in M_i\}$ is rainbow by choice of the $x_m$ vertices. This combined with Obervation~\ref{ObservationStarPlusMatching} and $\alpha n +1 +e(M_i)\leq \alpha n+\phi n$ show  that $D^1_i$ is a rainbow $(\leq \alpha n +\phi n)$-spider on $|D_i|+|M_i|\geq (1-4\phi-2\alpha)n$ vertices.
There are $e(M_i)$  colours on $D_i^1$ outside of $C_F$ (the colours on $M_i$). 
Therefore there are at most $e(D_i^1)-e(M_i)$ colours of $C_F$ on $D_i^1$, and hence at least $|C_F|-(e(D_i^1)-e(M_i))= 
|C_F|-e(D_i^1)+(n-|C_F|-1)
=n-|D_i^1|$ colours of $C_F$ outside $D_i^1$. This shows that condition (i) of Lemma~\ref{LemmaDisjointSpanningDoubleStars} holds with $\alpha=\alpha$, $\gamma=\alpha+\phi$, $\tau=4\phi+2\alpha$, and $\epsilon=\epsilon$.
Therefore, we can  apply Lemma~\ref{LemmaDisjointSpanningDoubleStars} to get the required spiders.

Now we have   edge-disjoint spanning rainbow $(\leq 0.0006 n)$-spiders $D'_1, \dots, D'_{\alpha n}$ which are rooted at $r_1, \dots, r_{\alpha n}$ respectively. We can apply Lemma~\ref{LemmaChangeDoubleStarParameter} to these spiders one at a time to turn them into $t$-spiders.
At the $i$th application, let $G$ be $K_n$ minus all the spiders except $D_i$ and set $\delta=0.0006$. This way $N(v)\geq (1-\delta)|G|$ holds for the $n-|A|\geq (1-\delta)n$ vertices outside $A$, and so Lemma~\ref{LemmaChangeDoubleStarParameter} gives us a $t$-spider disjoint from all previously constructed spiders.
\end{proof}

\section{Concluding remarks}
Here we mention some interesting directions for further research. 

\subsubsection*{Improving the bounds}
The most natural open problem is to further improve the bounds on Conjectures~\ref{Conjecture_Brualdi_Hollongsworth} -- \ref{Conjecture_Constantine}. 
In this paper we limited ourselves to proving a good quantiative bound on the Brualdi-Hollingsworth Conjecture (Theorem~\ref{Theorem1Factorization}) and proving the strongest qualitative result (Theorem~\ref{TheoremProperColouring}).

Theorem~\ref{TheoremProperColouring} represents a simultaneous improvement to the best known bounds on Conjectures~\ref{Conjecture_Brualdi_Hollongsworth} -- \ref{Conjecture_Constantine}. If one wants to further improve the bounds on any one of these conjectures, then it is routine to modify our methods to do so.
Particularly, we mention that it is possible to obtain quite a good bound on Constantine's Conjecture by combining the proofs of Theorems~\ref{TheoremProperColouring} and~\ref{Theorem1Factorization}. This is because the source of the small constant ``$0.000001$'' in Theorem~\ref{TheoremProperColouring}  is that the colouring on $K_n$ was a general proper colouring (rather than a $1$-factorization). If instead we are in the setting of a $1$-factorization (as in Constantine's Conjecture) then it is easy to modify the proof to find around $0.01$ edge-disjoint spanning rainbow isomorphic trees. 
The big open problem seems to be to prove some sort of asymptotic version of Conjectures Conjectures~\ref{Conjecture_Brualdi_Hollongsworth} -- \ref{Conjecture_Constantine}.  For example does every properly $(n-1)$-edge-coloured $K_n$ have $(1-o(1))n$ edge-disjoint spanning rainbow trees?

\subsubsection*{Proper colourings versus bounded colourings}
A colouring of a graph is $b$-bounded if there are  at most $ b$ edges of each colour. Notice that every properly coloured $K_n$ is $n/2$-bounded. It would be interesting to know whether any of the results in this paper generalize to colourings which are bounded rather than proper. In this direction, the best result is by Carraher, Hartke, and  Horn~\cite{Carraher_Hartke_Horn} who showed that every $n/2$-bounded colouring of $K_n$ had $\lfloor n/1000\log n\rfloor$ edge-disjoint rainbow spanning trees.

Curiously, Theorem~\ref{TheoremProperColouring} is not true for colourings that are $n/2$-bounded. In fact, Sudakov and Volec~\cite{Sudakov_Volec} constructed $9$-bounded colourings of $K_n$ which contain no spanning rainbow tree of radius $2$. In particular this implies that there are $9$-bounded colourings of $K_n$ without any spanning rainbow spiders. This shows that if some analogue of Constantine's Conjecture holds for bounded colourings, then one would need to consider graphs different form spiders.

\subsubsection*{Finding copies of a rainbow tree}
Notice that Theorem~\ref{TheoremProperColouring} is qualitatively stronger than Conjecture~\ref{Conjecture_Constantine} --- Theorem~\ref{TheoremProperColouring} allows us to specify what spanning rainbow tree we find (whereas Conjecture~\ref{Conjecture_Constantine} only says that we should find isomorphic trees without specifying the isomorphism class of the trees). This opens up the intriguing area of what collections of rainbow trees can be found in every properly coloured $K_n$. In this direction one can modify the result in this paper to allow us to find several different spiders in a properly coloured $K_n$.

\begin{theorem}
Let $T_1, \dots, T_{0.000001n}$ be spiders on $n$ vertices with $T_i$ a $t_i$-spider for $0.003n\leq t_i\leq 0.2n$.
Then every properly coloured $K_n$  contains edge-disjoint rainbow spanning copies of $T_1, \dots, T_{0.000001 n}$.
\end{theorem}
The proof of the above theorem is identical to the proof of Theorem~\ref{TheoremProperColouring}, except that in the last line of the proof one applies Lemma~\ref{LemmaChangeDoubleStarParameter} to create $t_i$-spiders rather than $t$-spiders.

It would be interesting to know for what other collections of trees $T_1, \dots, T_{0.000001 n}$ the above theorem is true. This problem may be quite hard, since even for uncoloured complete graphs there are many open problems about finding edge-disjoint trees eg. the Gy\'arf\'as-Sumner Conjecture.

A related open problem is ``which rainbow trees can be found in every properly coloured $K_n$?''
At first glance, one might hope that for any $n$-vertex tree $T_n$, every properly coloured $K_n$ contains a rainbow copy of $T_n$. However this is false already for paths. Maamoun and Meyniel~\cite{Maamoun_Meyniel} found proper $(n-1)$-edge-colourings of $K_{n}$ without a spanning rainbow path. Some extensions of this result, showing that there are edge colourings not
containing some other spanning trees were found in \cite{BPS}.
On the other hand, together with Alon~\cite{Alon_Pokrovskiy_Sudakov} the authors showed how to find a rainbow path of length $n-o(n)$ in every properly edge-coloured $K_n$. Based on this, one can expect that perhaps for every tree $T$, a rainbow copy of $T$ is contained in every properly edge-coloured complete graph with a few more vertices than $T$. Indeed such a result was very  recently proved in \cite{MPS}.
 
\subsubsection*{Note added in proof} 
After this paper was written we learned that
very recently Balogh, Liu and Montgomery \cite{BLM} proved the existence of $\epsilon n$ edge-disjoint spanning rainbow trees in every properly edge-colored $K_n$.

\subsubsection*{Acknowledgment}
The authors are grateful to D\"om\"ot\"or P\'alv\"olgyy to for comments on this paper.

\end{document}